\documentclass[pdflatex,sn-mathphys-ay]{sn-jnl}

\usepackage{graphicx}%
\usepackage{multirow}%
\usepackage{amsmath,amssymb,amsfonts}%
\usepackage{amsthm}%
\usepackage{mathrsfs}%
\usepackage[title]{appendix}%
\usepackage{xcolor}%
\usepackage{textcomp}%
\usepackage{manyfoot}%
\usepackage{booktabs}%
\usepackage{algorithm}%
\usepackage{algorithmicx}%
\usepackage[noend]{algpseudocode}%
\usepackage{listings}%
\usepackage{tabulary}
\usepackage{pgfplots}
\usepackage{pgfplotstable} 
\usepackage{standalone}
\usepackage{subcaption}
\pgfplotsset{compat=1.8}

\usepackage{mathtools}
\usepackage{enumitem}

\newcommand{\RR}{\mathbb{R}}

\newcommand{\lb}[1]{\underline{#1}}
\newcommand{\ub}[1]{\overline{#1}}
\newcommand{\ib}[1]{\mathbf{#1}}
\newcommand{\imid}[1]{{#1}^c}
\newcommand{\irad}[1]{{#1}^\Delta}
\newcommand{\ubf}{\overline{f}_{\textnormal{fin}}}
\newcommand{\SD}{\mathcal{SD}}

\newcommand{\algorithmicbreak}{\textbf{break}}
\newcommand{\Break}{\State \algorithmicbreak}
\newcommand{\Not}{\textbf{not} }

\DeclarePairedDelimiter\abs{\lvert}{\rvert}%

\theoremstyle{thmstyleone}%
\newtheorem{theorem}{Theorem}
\newtheorem{corollary}[theorem]{Corollary}%

\theoremstyle{thmstyletwo}%

\theoremstyle{thmstylethree}%

\begin{document}

\title[Heuristics for the Worst Optimal Value of Interval TPs]{Heuristics for the Worst Optimal Value of Interval Transportation Problems}


\author*[1]{\fnm{Elif} \sur{Radová Garajová}}\email{elif.garajova@vse.cz} 

\author[1,2]{\fnm{Miroslav} \sur{Rada}}\email{miroslav.rada@vse.cz} 

\affil*[1]{\orgdiv{Department of Econometrics}, \orgname{Faculty of Informatics and Statistics, Prague University of Economics and Business}, \orgaddress{\street{n\'{a}m. W. Churchilla 4}, \city{Prague}, \postcode{13067}, \country{Czech Republic}}}

\affil[2]{\orgdiv{Department of Financial Accounting and Auditing}, \orgname{Faculty of Finance and Accounting, Prague University of Economics and Business}, \orgaddress{\street{n\'{a}m. W. Churchilla 4}, \city{Prague}, \postcode{13067}, \country{Czech Republic}}}


\abstract{An interval transportation problem represents a model for a~transportation problem in which the values of supply, demand, and transportation costs are affected by uncertainty and can vary independently within given interval ranges. One of the main tasks of solving interval programming models is computing the best and worst optimal value over all possible choices of the interval data. Although the best optimal value of an interval transportation problem can be computed in polynomial time, computing the worst (finite) optimal value was proved to be NP-hard. In this paper, we strengthen a~previous result showing a quasi-extreme decomposition for finding the worst optimal value, and building on the result, we design heuristics for efficiently approximating the value. Using a simplified encoding of the scenarios, we first derive a local search method and a genetic algorithm for approximating the worst optimal value. Then, we integrate the two methods into a memetic algorithm, which combines the evolutionary improvement of a genetic algorithm with individual learning implemented via local search. Moreover, we include numerical experiments for a practical comparison of the three different approaches. We also show that the proposed memetic algorithm is competitive with the available state-of-the-art methods for approximating the worst optimal value of interval transportation problems, this is demonstrated by finding the new best solutions for several instances, among others.}




    \maketitle
	
	\section{Introduction}\label{sec:Intro}
	
	Optimization under uncertainty has played a crucial role in modeling and solving practical problems of operations research. Interval programming~\citep{Rohn:IntervalLinearProgramming:2006} provides a~framework for handling uncertainty in problems, where input data can vary independently within given lower and upper bounds. In this paper, we discuss the interval transportation problem~\citep{Cerulli:BestWorstValues:2017,Garajova:IntervalTransportationProblem:2023}, in which the values of supply, demand, and the unit transportation costs are uncertain data with given interval ranges.
	
	Analogously to computing optimal values in linear programming, one of the main tasks of solving an interval linear program is to find the best and the worst optimal value over all possible choices of the interval data, also known as the optimal value range problem~\citep{Mohammadi:IntervalLinearProgramming:2020}. In this paper, we address the problem of computing the bounds of the optimal value range restricted to feasible scenarios, which is an interesting task for interval programs that have infeasible scenarios making the worst bound of the traditionally defined optimal value range infinite. This leads to the task of computing the worst finite optimal value, which was also previously studied in the literature on interval linear programming~\citep{Hladik:WorstCaseFinite:2018}.
	
	In the case of interval transportation problems, the task of computing the best optimal value can be reduced to solving a single linear program with fixed data~\citep{Liu:TotalCostBounds:2003}. However, the decision problem connected with the task of computing the worst finite optimal value was proved to be NP-hard for all formulations of interval transportation problems~\citep{Garajova:ComplexityComputingWorst:2024,Hoppmann-Baum:ComplexityComputingMaximum:2022}. Therefore, efficient methods for approximating the value or solving special cases of interval transportation problems are also of interest. Various methods for tackling the problem have been proposed in the literature, including a nonlinear formulation based on duality by \cite{Liu:TotalCostBounds:2003}, a heuristic method by \cite{Juman:HeuristicSolutionTechnique:2014}, a permutation heuristic genetic algorithm by \cite{Xie:UpperBoundMinimal:2017}, an iterated local search by \cite{Cerulli:BestWorstValues:2017} and a mixed-integer linear programming formulation by \cite{Garajova:IntervalTransportationProblem:2023}. Specialized methods for the subclass of interval transportation problems that are immune against the more-for-less paradox were studied by \cite{Carrabs:ImprovedHeuristicApproach:2021a} and \cite{DAmbrosio:OptimalValueRange:2020}.
	
	Here, we build on a recent result showing a reduction for computing the worst finite optimal value using a finite number of so-called quasi-extreme scenarios, which was originally derived for problems immune against the more-for-less transportation paradox~\citep{Carrabs:ImprovedHeuristicApproach:2021a} and later generalized to interval transportation problems without any further assumption~\citep{Garajova:QuasiextremeReductionInterval:2024}. We first strengthen the result and show that it is sufficient to limit the reduction to balanced quasi-extreme scenarios (with equal total supply and total demand). Then, we define a suitable encoding of the scenarios and formulate a local search for finding a quasi-extreme scenario with a high optimal value via searching the neighborhood formed by slightly modified balanced scenarios. Moreover, we also adapt the framework of genetic algorithms~\citep{Sivanandam:IntroductionGeneticAlgorithms:2008,Xie:UpperBoundMinimal:2017} to provide an alternative heuristic approach for computing the worst optimal value using a population-based evolutionary process on the set of encoded quasi-extreme scenarios. Finally, we combine the two methods to design a memetic algorithm~\citep{Moscato:ModernIntroductionMemetic:2010}, in which the evolutionary improvement is augmented by individual learning through local search. We also compare the efficiency of the proposed methods and quality of the solutions produced by the different approaches in a~numerical experiment.
	
	\section{Interval Transportation Problem}\label{sec:ITP}
	Let us now formally introduce the mathematical model of the interval transportation problem and the relevant notation and terminology.
	
	Given two real matrices $\lb{C}, \ub{C} \in \RR^{m \times n}$ with $\lb{C} \le \ub{C}$, we define an \emph{interval matrix} with the \emph{lower bound} $\lb{C}$ and the \emph{upper bound} $\ub{C}$ as the set
	$
	\ib{C} = [\lb{C}, \ub{C}] = \{ C \in~\RR^{m \times n} : \lb{C} \le C \le \ub{C} \},
	$
	where all matrix inequalities are understood element-wise. An \emph{interval vector} is defined analogously.
	
	\subsection{Problem Formulation}\label{ssec:ITP:Formulation}
	Assume the following data are given:
	\begin{itemize}[topsep=0pt]
		\item a set of $m$ \emph{sources} denoted by $I = \{1, \dots, m\}$, such that each source $i \in I$ has a limited \emph{supply} $s_i$,
		\item a set of $n$ \emph{destinations} denoted by $J = \{1, \dots, n\}$, such that each destination $j \in J$ has a \emph{demand} $d_j$ to be satisfied,
		\item a \emph{unit transportation cost} $c_{ij}$ of transporting one unit of goods from source~$i$ to destination~$j$.
	\end{itemize}
	Moreover, assume that the values of supply $s_i$, demand $d_j$ and the unit costs $c_{ij}$ are uncertain quantities that can independently vary within given (non-negative) intervals $\ib{s}_i = [\lb{s}_i, \ub{s}_i]$, $\ib{d}_j = [\lb{d}_j, \ub{d}_j]$ and $\ib{c}_{ij} = [\lb{c}_{ij}, \ub{c}_{ij}]$, respectively. Note that an interval $[\lb{a}, \ub{a}]$ is said to be non-negative if all values within the interval are non-negative, i.e. $\lb{a} \ge 0$ holds.
	
	The objective of the transportation problem is to construct a transportation plan for shipping goods from the sources to the destinations while respecting the available supplies, satisfying the required demands and minimizing the total transportation costs. Here, we also take into account the uncertain nature of the input data.
	
	Using the framework of interval linear programming \citep{Rohn:IntervalLinearProgramming:2006}, we can model the \emph{interval transportation problem} as follows:
	\begin{align}\tag{ITP}\label{eq:itp:unbalanced}
		\begin{alignedat}{3}
			\text{minimize } 	& \quad & \omit\rlap{$\displaystyle \sum_{i \in I} \sum_{j \in J} [\lb{c}_{ij}, \ub{c}_{ij}] x_{ij}$} \\
			\text{subject to } 	&	& \sum_{j \in J} x_{ij} & \le [\lb{s}_i, \ub{s}_i], & \qquad & \forall i \in I,  \\
			&	& \sum_{i \in I} x_{ij} & = [\lb{d}_j, \ub{d}_j], 	&		& \forall j \in J, \\
			&	& x_{ij} 				& \ge 0, 					&		& \forall i \in I, j \in J,
		\end{alignedat}
	\end{align}
	where the real variables $x_{ij}$ represent the amount of goods transported from a~source $i \in I$ to a destination $j \in J$. The interval linear program \eqref{eq:itp:unbalanced} then corresponds to the set of all transportation problems in the same form with input data within the respective intervals. A particular transportation problem determined by a~choice of the coefficients from the intervals is called a \emph{scenario} of the interval transportation problem.
	
	\subsection{Optimal Values}\label{ssec:ITP:OptVal}
	One of the main goals of solving interval linear programs is to compute the optimal value range, which is the tightest interval enclosing the best and the worst optimal value over all possible scenarios (usually including infinite values for infeasible and unbounded scenarios). The problem of calculating the bounds of the traditionally defined optimal value range has been well studied in the interval programming literature~\citep{Mohammadi:IntervalLinearProgramming:2020}.
	
	Here, we consider the problem of computing the optimal value range of \eqref{eq:itp:unbalanced} restricted to the set of feasible scenarios, which leads to the task of finding the worst finite optimal value \citep{Hladik:WorstCaseFinite:2018}. It is easy to see that a particular transportation problem in \eqref{eq:itp:unbalanced} has a feasible solution if and only the total available supply is sufficient to satisfy the total demand. Let us denote by $\SD$ the set of all supply-demand vectors $(s,d) \in (\ib{s}, \ib{d})$ corresponding to feasible scenarios, i.e.
	\begin{align}
		\SD = \bigg\{(s,d) \in (\ib{s}, \ib{d}) : \sum_{i\in I} s_i \ge \sum_{j \in J} d_j \bigg\}.
	\end{align}
	
	Let $f(C,s,d)$ denote the optimal value of a scenario with a cost matrix $C$, supply vector $s$ and a demand vector $d$. Then, we can define the \emph{best optimal value} $\lb{f}$ and the \emph{worst finite optimal value} $\ubf$ as
	\begin{align}
		\lb{f}(\ib{C}, \ib{s}, \ib{d}) &= \min\, \{ f(C,s,d) : (s,d) \in \SD,\ C \in \ib{C} \},\\
		\ubf(\ib{C}, \ib{s}, \ib{d}) &= \max\, \{ f(C,s,d) : (s,d) \in \SD,\ C \in \ib{C} \}.\label{eq:ubf}
	\end{align}
	For interval linear programs with non-negative variables, the best optimal value $\lb{f}$ can be easily computed by solving a single linear program (see also \citep{Liu:TotalCostBounds:2003} for details regarding interval transportation problems). On the other hand, the problem of computing the worst finite optimal value was proved to be NP-hard, both in general and for the case of interval transportation problems \citep{Garajova:ComplexityComputingWorst:2024,Hoppmann-Baum:ComplexityComputingMaximum:2022}.
	
    For the purpose of this paper, we assume that $\SD \neq \emptyset$ (there exists at least one feasible scenario) and $\SD \neq (\ib{s}, \ib{d})$ (there exists at least one infeasible scenario). 
	Note that these are the most challenging cases for computing the worst finite optimal value, since in other cases the problem can be solved in polynomial time (either there is no feasible scenario, or all scenarios are feasible and $\ubf$ is attained for $(\ub{C}, \lb{s}, \ub{d})$, see also \citep{Xie:UpperBoundMinimal:2017}).
	
	\subsection{Quasi-extreme Reduction}\label{ssec:ITP:QuasiReduct}
	A scenario with a supply-demand vector $(s,d) \in (\ib{s}, \ib{d})$ is called \emph{supply-quasi-extreme}, if there exists at most one source $k \in I$ such that
	\begin{align*}
		&s_i \in \{\lb{s}_i, \ub{s}_i\} \text{ for each } i \in I \setminus \{k\}, \qquad \lb{s}_k \le s_k \le \ub{s}_k, \text{ and}\\
		&d_j \in \{\lb{d}_j, \ub{d}_j\} \text{ for each } j \in J.
	\end{align*}
	Similarly, a scenario with a supply-demand vector $(s,d) \in (\ib{s}, \ib{d})$ is called \emph{demand-quasi-extreme}, if there exists at most one destination $k \in J$ such that
	\begin{align*}
		&d_j \in \{\lb{d}_j, \ub{d}_j\} \text{ for each } j \in J\setminus\{k\}, \qquad \lb{d}_k \le d_k \le \ub{d}_k, \text{ and}\\
		&s_i \in \{\lb{s}_i, \ub{s}_i\} \text{ for each } i \in I.
	\end{align*}
	We refer to the supply-quasi-extreme and demand-quasi-extreme scenarios collectively as \emph{quasi-extreme scenarios}. Moreover, we refer to the value $s_k$ or $d_k$, which is not necessarily set to one of the bounds, as the \emph{free value} of the quasi-extreme scenario.
	
	A reduction for computing $\ubf$ using quasi-extreme scenarios was originally derived for interval transportation problems immune against the more-for-less paradox~\citep{Carrabs:ImprovedHeuristicApproach:2021a} and later extended to general interval transportation problems without any further assumption \citep{Garajova:QuasiextremeReductionInterval:2024}. Theorem~\ref{thm:quasi:ubf} states the main result allowing the reduction.
	
	\begin{theorem}[\citet{Garajova:QuasiextremeReductionInterval:2024}, \citet{Carrabs:ImprovedHeuristicApproach:2021a}]\label{thm:quasi:ubf}
		There exists a quasi-extreme scenario for which the worst finite optimal value $\ubf$ of interval transportation problem \eqref{eq:itp:unbalanced} is attained.
	\end{theorem}
	
	The set of all quasi-extreme scenarios is not finite, in general, since the free supply value~$s_k$ or the free demand value~$d_k$ can be set to any value from the corresponding interval. However, Theorem~\ref{thm:quasi:fixLast} shows that for the purpose of computing~$\ubf$, a~suitable value of $s_k$ or $d_k$ can also be calculated based on the other values of supply and demand. Moreover, note that because the variables are non-negative, the cost matrix~$\ub{C}$ yields the worst optimal value.
	
	\begin{theorem}[\citet{Garajova:QuasiextremeReductionInterval:2024}]\label{thm:quasi:fixLast}
		Let the worst optimal value $\ubf(\ub{C}, \ib{s}, \ib{d})$ be attained for a quasi-extreme scenario with a supply-demand vector $(s,d) \in \SD$ and let $k$ be the index of the free value. Then, $\ubf$ is also attained for a scenario with the supply-demand vector $(s', d')$ where
		\[
		s'_k = \max \left\{ \lb{s}_k,\ \sum_{j \in J} d_j - \hskip-7pt\sum_{i \in I\backslash \{k\}}\hskip-8pt s_i\right\},  \quad s'_i = s_i \text{ for } i \in I\backslash\{k\}, \quad d' = d
		\]
		for a supply-quasi-extreme scenario $(\ub{C},s,d)$, or,
		\[
		d'_k = \min \left\{ \ub{d}_k,\  \sum_{i \in I} s_i -\hskip-7pt\sum_{j \in J\backslash \{k\}}\hskip-8pt d_j \right\}, \quad d'_j = d_j \text{ for } j \in J\backslash\{k\}, \quad s' = s
		\]
		for a demand-quasi-extreme scenario $(\ub{C},s,d)$.
	\end{theorem}
	
	A scenario $(C, s, d)$ is said to be \emph{balanced} if the total supply is equal to the total demand, i.e. $\sum_{i \in I} s_i = \sum_{j \in J} d_j$ holds.
	We can further strengthen the result of Theorem~\ref{thm:quasi:fixLast} and prove that exploring the unbalanced scenarios is not necessary, since the value~$\ubf$ will always be attained for a balanced one.
	
	\begin{theorem}\label{thm:quasi:fixLast:balanced}
		Consider a feasible quasi-extreme scenario $(\ub{C}, s, d)$ with a supply-demand vector $(s,d) \in \SD$ and let $k$ be the index of the free value.
		
		If the scenario is supply-quasi-extreme and 
		\begin{equation}
			\lb{s}_k > \sum_{j \in J} d_j - \hskip-7pt\sum_{i \in I\backslash \{k\}} \hskip-7pt s_i \label{eq:quasi:fix:Sk}
		\end{equation}
		holds, or, if the scenario is demand-quasi-extreme and 
		\begin{equation}
			\ub{d}_k < \sum_{i \in I} s_i - \hskip-7pt\sum_{j \in J\backslash \{k\}}\hskip-7pt d_j
		\end{equation}
		holds, then there exists a quasi-extreme scenario $(\ub{C}, s', d')$ with $\sum_{i\in I} s'_i = \sum_{j \in J} d'_j$ such that $f(\ub{C}, s, d) \le f(\ub{C}, s', d')$.
	\end{theorem}
	\begin{proof}
		First, assume that the scenario $(\ub{C}, s, d)$ from the statement of the theorem is supply-quasi-extreme, with a free value $s_k$. Consider a modified completely extreme scenario with a supply vector \[ s^* = (s_1, \dots, s_{k-1}, \lb{s}_k, s_{k+1}, \dots, s_m)^T \] and the same demand vector $d$. Replacing the value $s_k$ with the (lower or equal) value $\lb{s}_k$ in the scenario can only lead to an increase in the optimal value, that is, we have $f(\ub{C}, s, d) \le f(\ub{C}, s^*, d)$. Furthermore, by inequality \eqref{eq:quasi:fix:Sk}, the modified scenario is still feasible.
		
		From the assumption that there exists an infeasible scenario we know that $\sum \lb{s}_i \le \sum \ub{d}_j$ holds. Therefore, it is possible to sequentially decrease the remaining supply values down to the lower bounds, and, if necessary, sequentially increase the demand values up to the upper bounds, until we obtain a scenario $(\ub{C}, s', d')$ with $\sum s'_i = \sum d'_j$. By construction, the obtained scenario is quasi-extreme and obviously also feasible. Moreover, decreasing the supply with fixed demand or increasing the demand for fixed supply can only lead to a worse optimal value (see \citep{Cerulli:BestWorstValues:2017} for details). Therefore, we obtain the desired inequality
		\[
		f(\ub{C}, s', d') \ge f(\ub{C}, s^*, d) \ge f(\ub{C}, s, d).
		\]
		The latter case with a demand-quasi-extreme scenario can be proved analogously.
	\end{proof}
	
	The following corollary states a direct consequence of Theorem~\ref{thm:quasi:fixLast:balanced}, which yields an exact finite algorithm for computing $\ubf$ over balanced quasi-extreme scenarios.
	
	\begin{corollary}
		The worst finite optimal value $\ubf$ of interval transportation problem \eqref{eq:itp:unbalanced} can be computed as the maximum over the optimal values of all balanced quasi-extreme scenarios.
	\end{corollary}
	
	Algorithm~\ref{alg:exactQuasi} summarizes the method for computing $\ubf$ exactly by examining at most $(m+n)2^{m+n-1}$ relevant balanced quasi-extreme scenarios.


	\begin{algorithm}
		\begin{algorithmic}
			\State worstVal $\gets -\infty$
			\For{$k \in I$}
			\For{all choices of $s_i \in \{\lb{s}_i, \ub{s}_i\}$ for $i \in I \setminus \{k\}$ and $d_j \in \{\lb{d}_j, \ub{d}_j\}$ for $j \in J$}
			\If{$\sum_{j\in J} d_j - \sum_{i\in I\backslash\{k\}} s_i > \ub{s}_k$}
			\State Scenario $(\ub{C},s,d)$ is infeasible for all $s_k \in \ib{s}_k$ and can be ignored
			\ElsIf{$\lb{s}_k > \sum_{j \in J} d_j - \sum_{i \in I\backslash \{k\}} s_i$}
			\State Scenario $(\ub{C},s,d)$ is unbalanced and can be ignored
			\ElsIf{$f(\ub{C},s,d) > $ worstVal}
			\State worstVal $\gets f(\ub{C},s,d)$ 
			\State worstScen $\gets (\ub{C},s,d)$
			\EndIf
			\EndFor
			\EndFor
			\For{$k \in J$}
			\For{all choices of $s_i \in \{\lb{s}_i, \ub{s}_i\}$ for $i \in I$ and $d_j \in \{\lb{d}_j, \ub{d}_j\}$ for $j \in J \setminus \{k\}$}
			\If{$\sum_{i\in I} s_i - \sum_{j\in J\backslash\{k\}} d_j < \lb{d}_k$}
			\State Scenario $(\ub{C},s,d)$ is infeasible for all $d_k \in \ib{d}_k$ and can be ignored
			\ElsIf{$\ub{d}_k < \sum_{i \in I} s_i - \sum_{j \in J\backslash \{k\}} d_j$}
			\State Scenario $(\ub{C},s,d)$ is unbalanced and can be ignored
			\ElsIf{$f(\ub{C},s,d) > $ worstVal}
			\State worstVal $\gets f(\ub{C},s,d)$ 
			\State worstScen $\gets (\ub{C},s,d)$
			\EndIf
			\EndFor
			\EndFor
			\Return worstVal, worstScen
		\end{algorithmic}
		\caption{A quasi-extreme reduction for computing $\ubf$ of \eqref{eq:itp:unbalanced}}\label{alg:exactQuasi}
	\end{algorithm}
	
	\section{Heuristics for the Interval Transportation Problem}\label{sec:Alg}
    In this section, we derive heuristic approaches for finding a lower bound on the worst optimal value $\ubf$ based on the quasi-extreme reduction. First, let us define a simplified encoding of the quasi-extreme scenarios.
    
	Let $\chi$ denote the set of all vectors $(a_1, \dots, a_{m+n}) \in \{-1,0,1\}^{m+n}$ with exactly one $a_i = 0$.
	Given an interval transportation problem, each vector $a \in \chi$ can be associated with a quasi-extreme scenario $(\ub{C}, s, d)$, where $a_i = -1$ corresponds to the values set to the lower bound of the respective interval, $a_i = 1$ corresponds to the values set to the upper bound and the free value $s_k$ or $d_k$ is calculated as described in Theorem~\ref{thm:quasi:fixLast} (or, if there is no choice of the free value, which would make the scenario feasible, we set $s_k = \ub{s}_k$ or $d_k = \lb{d}_k$). Formally, for a given $a \in \chi$, we have a~supply-demand vector $(s,d)$ defined as follows:
	\begin{align*}
		&\forall i \in \{1, \dots, m\} \colon s_i = \imid{s}_i + a_i  \irad{s}_i &&\text{if } a_i \neq 0,\\
		&\forall j \in \{1, \dots, n\} \colon d_j = \imid{d}_j + a_{m+j}  \irad{d}_j &&\text{if } a_{m+j} \neq 0,\\
		&s'_k = \max \left\{ 
		\lb{s}_k,\ 
		\min \left\{ 
		\sum_{j \in J} d_j - \hskip-7pt\sum_{i \in I\backslash \{k\}}\hskip-8pt s_i, 
		\ub{s}_k 
		\right\}\right\}, 
		&& \text{if } a_k = 0 \text{ with } k \in \{1, \dots, m\},\\
		&d'_k = \min \left\{ 
		\ub{d}_k,\  
		\max \left\{
		\sum_{i \in I} s_i -\hskip-7pt\sum_{j \in J\backslash \{k\}}\hskip-8pt d_j,
		\lb{d}_k 
		\right\}\right\},
		&& \text{if } a_{m+k} = 0 \text{ with } k \in \{1, \dots, n\}.
	\end{align*}
	Therefore, the set $\chi$ introduced above encodes all quasi-extreme scenarios relevant to computing the worst finite optimal value $\ubf$ of a given interval transportation problem. In the following, we describe heuristic algorithms for approximating $\ubf$, which work in the search space defined by the set $\chi$ (or its subset). Here, we define the value of a given $a \in \chi$ encoding a supply-demand vector $(s, d)$ as the optimal value of the scenario $(\ub{C}, s, d)$, i.e. we have $f(a) = f(\ub{C},s,d)$.


	\subsection{Local Search Algorithm}\label{ssec:Alg:LocalSearch}

	In this section, we derive a local search method for approximating the worst finite optimal value of an interval transportation problem by exploring a suitable subset of the configurations in $\chi$. Consider a configuration $(a_1, \dots, a_{m+n}) \in \chi$ with \mbox{$a_k = 0$} encoding a balanced quasi-extreme scenario. For an index $i \neq k$, the following modification is used to obtain a neighboring configuration: 
    \begin{itemize}
    \item We attempt to switch the value~$a_i$ to the other bound, while modifying the free value corresponding to $a_k$ to keep the scenario balanced. 
    \item If such modification is not possible, the value corresponding to $a_k$ is set to one of the bounds and the value corresponding to $a_i$ becomes the free value instead.
    \end{itemize}

	Let $\gamma^k_i(a)$ denote such perturbation of the $i^\text{th}$ coefficient of the scenario encoded by $a$. Then, we define the \emph{neighborhood} $\mathcal{N}(a)$ of a given configuration $a \in \chi$ with $a_k = 0$ as the set
	\[
	\mathcal{N}(a) = \{ \gamma^k_i(a) : i \in \{1, \dots, m+n\} \backslash \{k\} \}.
	\]
	
	Let us now summarize the framework of the local search heuristic for computing the worst finite optimal value of an interval transportation problem. In each step of the algorithm, a neighboring configuration is chosen based on the optimal value of the corresponding scenario. Algorithm \ref{alg:localSearch:first} utilizes the first improvement policy, where the first encountered neighboring configuration encoding a scenario with a lower optimal value $f(\ub{C}, s, d)$ than the current one is chosen. 
	
	\begin{algorithm}[h!]
		\begin{algorithmic}
			\State $a \gets$ initial feasible quasi-extreme scenario
			\Repeat
			\State improved $\gets$ False
			\For{$a' \in \mathcal{N}(a)$}
			\If{$f(a') > f(a)$}
			\State $a \gets a'$
			\State improved $\gets$ True
			\Break
			\EndIf
			\EndFor
			\Until{\Not improved}\\
			\Return $a$
		\end{algorithmic}
		\caption{A first-improvement local search for computing $\ubf$ of \eqref{eq:itp:unbalanced}}\label{alg:localSearch:first}
	\end{algorithm}
	
	On the other hand, Algorithm \ref{alg:localSearch:best} employs the (more expensive) best improvement policy, where all neighboring configurations are explored and the quasi-extreme scenario with the lowest optimal value is used to select the following configuration.
	
	\begin{algorithm}[h]
		\begin{algorithmic}
			\State $a \gets$ initial feasible quasi-extreme scenario
			\Repeat
			\State improved $\gets$ False
			\State neighValue $\gets -\infty$
			\For{$a' \in \mathcal{N}(a)$}
			\If{$f(a') > $ neighValue}
			\State neighValue $\gets f(a')$
			\State neigh $\gets a'$
			\EndIf
			\EndFor
			\If{neighValue $> f(a)$}
			\State $a \gets$ neigh
			\State improved $\gets$ True
			\EndIf
			\Until{\Not improved}\\
			\Return $a$
		\end{algorithmic}
		\caption{A best-improvement local search for computing $\ubf$ of \eqref{eq:itp:unbalanced}}\label{alg:localSearch:best}
	\end{algorithm}
	
	\subsection{Genetic Algorithm}\label{ssec:Alg:Genetic}
	Genetic algorithms~\citep{Sivanandam:IntroductionGeneticAlgorithms:2008} provide a population-based approach to constructing heuristics for challenging optimization problems inspired by natural selection. The main building blocks of a genetic algorithm are the operators for implementing selection (based on the evaluation of a fitness function), crossover, and mutation that can be applied to individuals of a given population within an iterative evolutionary process.
	
	Here, we adapt the framework of genetic algorithms to the problem of computing the worst optimal value $\ubf$ of interval transportation problems. In our case, the population corresponds to the set of configurations $\chi$ representing quasi-extreme scenarios introduced in Section~\ref{sec:Alg}. The size of the population is determined by a parameter $N_\text{pop}$.
	
	Algorithm~\ref{alg:geneticAlg} presents the framework of the genetic algorithm designed to compute $\ubf$. The parameters $\pi_\text{C}$ and $\pi_\text{M}$ determine the probability of applying the crossover or mutation operator, respectively. Given a parameter $t_\text{GA}$, the termination criterion for the iterative evolutionary process is that no improvement in maximal fitness (optimal value of a scenario) was observed in the population in previous $t_\text{GA}$ iterations. At termination, the configuration with the highest optimal value found is returned. 
	
	\begin{algorithm}
		\begin{algorithmic}
			\State $P \gets $ Population of $N_\text{pop}$ randomly generated feasible configurations from $\chi$
			\State worstVal $\gets -\infty$
			\Repeat
            \State $Q \gets \emptyset$
			\For{$a \in P$}
			\State Evaluate fitness $f(a)$
			\If{$f(a) >$ worstVal}
			\State worstVal $\gets f(a)$
			\State worstConf $\gets a$
			\EndIf
			\EndFor
            \For{$l \in \{1, \dots, N_{\text{pop}}\}$}
			\State Choose $a \in P$ to copy to $Q$ according to the selection strategy
			\EndFor
			\State Split $Q$ into $\abs{Q}/2$ pairs
			\For{each created pair $(a, b)$ from $Q$}
			\State Insert crossover($a$, $b$) into $Q$ with probability $\pi_\text{C}$
			\EndFor
			\For{$a \in Q$}
			\State Replace $a$ with mutate($a$) with probability $\pi_\text{M}$
			\EndFor
			\State $P \gets Q$
			\Until no change of worstVal occurs for $t_\text{GA}$ iterations\\
			\Return worstConf
		\end{algorithmic}
		\caption{A genetic algorithm for computing $\ubf$ of \eqref{eq:itp:unbalanced}}\label{alg:geneticAlg}
	\end{algorithm}

    \subsubsection{Fitness and Selection}\label{ssec:genetic:fitness}
	The fitness of an individual configuration $a \in \chi$ is determined by the optimal value $f(a) = f(\ub{C},s,d)$ of the transportation problem represented by the configuration. If a given configuration corresponds to an infeasible scenario, it is assigned the optimal value of a feasible scenario with a similar structure (e.g. made by iteratively switching the values of supplies from the lower bound to the upper bound and inversely for the values of demands, until a feasible scenario is reached).
    
    We consider two selection policies, in which $N_{\text{pop}}$ configurations from the current population are chosen to advance to the following generation. The algorithm uses:
    \begin{itemize}
    \item a fitness proportionate selection strategy, in which each configuration $a$ has a given probability depending on $f(a)$ of being selected, or,
    \item a tournament selection strategy, in which a fixed number of configurations are chosen and the configuration with the highest fitness value $f(a)$ among them is selected.
    \end{itemize}
    Additionally, a fixed number of configurations with the highest fitness value $f(a)$ can be selected to be preserved in advance. 

    For the fitness proportionate selection strategy, the following functions can be used to compute the probability of selection for a configuration $a$:
   \[
        \pi_S^1(a) = \frac{f(a)}{\sum_{b \in P} f(b)}, \quad \pi_S^2(a) = \frac{f(a)-f_{\min}}{\sum_{b \in P} f(b) - \abs{P}\cdot f_{\min}}, \quad \pi_S^3(a) = \frac{f(a)-g}{\sum_{b \in P} f(b) - \abs{P} g},
   \]
    where $f_{\min}$ denotes the smallest value of fitness in the population, and the value $g$ is calculated such that the ratio of the highest and the lowest value of $\pi_S^3$ is constant.
    
    \subsubsection{Mutation}

    Given a configuration $a \in \chi$ with $a_k = 0$, the operator mutate($a$) produces a new configuration $a'$ by either switching a coefficient of $a$ or by changing the index of the free value, depending on whether the corresponding scenario of the transportation problem is balanced.

    For a configuration corresponding to an unbalanced scenario, the mutation operator randomly chooses an index $i \in \{1, \dots, m+n\} \setminus \{k\}$ and sets $a'_i = -a_i$, while preserving the remaining coefficients.

    For a balanced scenario, the mutation operator attempts to change the index of the free value. Here, we set $a'_k \in \{-1, 1\}$, preferring the value that can lead to a balanced scenario after selecting a suitable index of the new free value. If creating a new balanced scenario is possible, we set $a'_i = 0$ for the suitable index $i$ balancing the change in $a_k$, while preserving the remaining coefficients. 
    
    If neither of the choices $a'_k \in \{-1, 1\}$ allows a balanced scenario, then there exists a coefficient $a_i$ such that setting $a'_k = 0$, $a'_i = -a_i$ and $a'_j = a_j$ for all $j \notin \{i, k\}$ corresponds to a balanced scenario. In this case, the mutated configuration $a'$ is defined in this way.

    Note that while the mutation operator can produce a configuration corresponding to an unbalanced or infeasible scenario, a balanced scenario is always mutated into a new balanced scenario.
    
    \subsubsection{Crossover}
    
	The operator crossover($a$, $b$) combines two parent configurations $a, b \in \chi$ to produce an offspring configuration $z \in \chi$. Here, each value of $z$ is inherited from $a$ or from $b$. Moreover, the position of the free value is also inherited from one of the parents. Assume that the free values of the parents are $a_k = 0$ and $b_l = 0$. Then, the offspring configuration $z$ is created as follows: 
	\begin{itemize}
		\item  Set randomly either $z_k = 0$ and $z_l = a_l$, or $z_l = 0$ and $z_k = b_k$.
		\item For each $i \neq k, l$, set randomly either $z_i = a_i$ or $z_i = b_i$.
	\end{itemize}

    Note that the crossover operator can generate configurations corresponding to unbalanced and infeasible scenarios, even if both parents were configurations corresponding to balanced scenarios.
	
	\subsection{Memetic Algorithm}\label{ssec:Alg:Memetic}
	Hybridized genetic algorithms that combine the framework of genetic algorithms with other heuristic approaches have also been proposed and studied in the literature. The idea of improving the evolutionary population-based approach with individual learning has led to the combination of genetic algorithms with local search methods, known in the literature as memetic algorithms.
	
	Algorithm~\ref{alg:memeticAlg} shows the framework of a memetic algorithm to compute the worst finite optimal value of the interval transportation problems, which is based on the genetic algorithm derived in Section~\ref{ssec:Alg:Genetic}. However, here the local search approach discussed in Section~\ref{ssec:Alg:LocalSearch} is also used, to improve the quality of selected individual configurations in the population (with the probability of selection determined by a parameter $\pi_\text{LS}$). 
	
	Individual learning through local search is used for the initial population, as well as for the offspring configurations created by the crossover operator, using the first improvement policy (see also Algorithm~\ref{alg:localSearch:first} and Section~\ref{ssec:Alg:LocalSearch} for details). In order to keep the computation time reasonable, the function LocalSearch($a$) in Algorithm~\ref{alg:memeticAlg} can also be limited to return the configuration found after at most $t_{LS}$ iterations of the local search algorithm.
	
	\begin{algorithm}
		\begin{algorithmic}
			\State $P \gets $ Population of $N_\text{pop}$ randomly generated configurations from $\chi$
			\For{$a \in P$}
			\State Replace $a$ in $P$ with LocalSearch($a$) with probability $\pi_{LS}$
			\EndFor
			\State worstVal $\gets -\infty$
			\Repeat
            \State $Q \gets \emptyset$
			\For{$a \in P$}
			\State Evaluate fitness $f(a)$
			\If{$f(a) >$ worstVal}
			\State worstVal $\gets f(a)$
			\State worstConf $\gets a$
			\EndIf
			\EndFor
			\For{$l \in \{1, \dots, N_{\text{pop}}\}$}
			\State Choose $a \in P$ to copy to $Q$ according to the selection strategy
			\EndFor
			\State Split $Q$ into $\abs{Q}/2$ pairs
			\For{each created pair $(a, b)$ from $Q$}
			\State Insert $z$ := crossover($a$, $b$) into $Q$ with probability $\pi_\text{C}$
			\If{$z$ was inserted into $Q$}
			\State Replace $z$ in $Q$ with LocalSearch($z$) with probability $\pi_\text{LS}$
			\EndIf
			\EndFor
			\For{$a \in Q$}
			\State Replace $a$ with mutate($a$) with probability $\pi_\text{M}$
			\EndFor
			\State $P \gets Q$
			\Until no change of worstVal occurs for $t_\text{GA}$ iterations\\
			\Return worstConf
		\end{algorithmic}
		\caption{A memetic algorithm for computing $\ubf$ of \eqref{eq:itp:unbalanced}}\label{alg:memeticAlg}
	\end{algorithm}
	
	\section{Computational Experiments}\label{sec:Experiment}
    We performed computational experiments on a collection of interval transportation problem instances to compare the computational efficiency and solution quality of existing approaches with those of the algorithms presented in this paper.
    
    \subsection{Implementation Details}
    The proposed methods were implemented in Python, using the Gurobi 12.0.1 solver for evaluating the scenarios and computing the fitness function. The source code of the implementation is available on GitHub \citep{Rada:Github:2025}.

    The state-of-the-art methods for approximating the worst finite optimal value of an interval transportation problem that are considered in the experiment are the permutation heuristic genetic algorithm by \cite{Xie:UpperBoundMinimal:2017}, the iterated local search algorithm by \cite{Cerulli:BestWorstValues:2017} and the mixed-integer linear programming formulation by \cite{Garajova:IntervalTransportationProblem:2023}. For the former two methods, the results previously published in the literature are used for reference, while the last model is solved using Gurobi 12.0.1.

    For the comparison of the different methods, the memetic algorithm was used with the tournament selection strategy and the parameters were set empirically as follows:
    \begin{itemize}
    \item size of the population $N_{\text{pop}} = 30$,
    \item number of non-improved iterations to terminate $t_{\text{GA}} = 20$,
    \item probability of local search $\pi_{\text{LS}} = 0.7$,
    \item no limit on the number of local search iterations $t_{\text{LS}}$,
    \item probability of mutation $\pi_{\text{M}} = 0.1$ for configurations corresponding to the balanced scenarios and $\pi_{\text{M}} = 0.7$ for the unbalanced scenarios,
    \item probability of crossover $\pi_{\text{C}} = 1$.
    \end{itemize}

    Complete specification of the parameter settings used in the experiments is available in the source code in the GitHub repository.

    The experiment was carried out on a computer with a 32 GB RAM and an Intel Core i7-1185G7 processor.

    \subsection{Instances}
    The dataset of interval transportation problem instances used in the experiments is a standard collection of problems introduced by \cite{Xie:UpperBoundMinimal:2017} and subsequently used by other authors \citep{Cerulli:BestWorstValues:2017, Garajova:IntervalTransportationProblem:2023}. The dataset originally contains 60 randomly generated benchmark instances split into groups of $10$ instances of each size $2 \times 3$, $3 \times 5$, $4 \times 6$, $5 \times 10$, $10 \times 10$ and $20 \times 20$ (here, the size refers to the number of sources and destinations). In all of these instances, the supply and demand interval vectors satisfy $\ub{s} = 2\lb{s}$ and $\ub{d} = 2 \lb{d}$.
    In this experiment, only the largest instances of size $20 \times 20$ are used, since the smaller instances can be efficiently solved by exact algorithms in negligible time (see also the results published by \cite{Garajova:IntervalTransportationProblem:2023}).

    The experiments with larger randomly generated instances of sizes $40\times40$ and $60\times60$ \citep[available on GitHub]{Rada:Github:2025} did not prove fruitful. The best solutions found are the same for all the compared algorithms and are structurally simple (most or even all supplies and demands are set to their upper bounds). These solutions are likely to be found by simple local search. There remains a tempting open auxiliary problem: how to generate ``hard'' instances with nontrivial structure of maximizers.

    \subsection{Results and Discussion}

    \subsubsection{Comparison of the proposed methods}
    Table~\ref{tab:exp:memetics} shows the results of the computational experiment when solving the instances with the four algorithms proposed in this paper: the local search algorithm with the first improvement policy (Alg.~\ref{alg:localSearch:first}) and with the best improvement policy (Alg.~\ref{alg:localSearch:best}), the genetic algorithm (Alg.~\ref{alg:geneticAlg}) and the hybrid memetic algorithm (Alg.~\ref{alg:memeticAlg}). For each of the methods, the objective value providing a lower bound on the worst finite optimal value of the interval transportation problem is reported, as well as the running time (in seconds) required to compute the value. 
    
    Note that while the local search algorithm has the fastest running time, it was not always able to find the highest objective value. On the other hand, the genetic algorithm generally found better solutions, but required a significantly longer running time. Finally, the memetic algorithm combines the strengths of both methods and strikes a balance between producing high-quality solutions (thanks to the population-based approach) and keeping the computation time reasonable (requiring less generations thanks to local search).

     Since the implemented methods are not deterministic, we also studied the variations in the obtained solutions and running time over several runs. Figure~\ref{fig:exp:intplot:obj} shows an interval plot of the minimal, maximal and average objective value obtained over $5$ runs of each of the algorithms \ref{alg:localSearch:first}, \ref{alg:localSearch:best}, \ref{alg:geneticAlg} and \ref{alg:memeticAlg}. We can see that the variation in the attained objective values is the highest for the local search methods, which depend highly on the randomly generated starting configuration, whereas the population-based methods are more stable. Similarly, Figure~\ref{fig:exp:intplot:time} shows the minimal, maximal and average computation times over the $5$ runs. Here, the plots highlight the efficiency of the local search methods, while the running time varies more for the population-based methods. The genetic algorithm requires the longest time to produce the solution in most cases, however, the memetic algorithm was able to find the solutions in all runs under $10$ seconds.

        \newcolumntype{K}[0]{>{\centering\arraybackslash}p{55pt}}
    \newcolumntype{M}[0]{>{\centering\arraybackslash}p{15pt}}

    \begin{table}[b]
    \begin{tabular}{cc@{\hspace{3pt}}cc@{\hspace{3pt}}cc@{\hspace{3pt}}cc@{\hspace{3pt}}cc@{\hspace{3pt}}c}
    \hline
    & \multicolumn{2}{c}{Local search (Alg. \ref{alg:localSearch:first})} & \multicolumn{2}{c}{Local search (Alg. \ref{alg:localSearch:best})} & \multicolumn{2}{c}{Genetic alg. (Alg. \ref{alg:geneticAlg})} & \multicolumn{2}{c}{Memetic alg. (Alg. \ref{alg:memeticAlg})} \\ 
& Obj. value & Time &		Obj. value & Time &		Obj. value & Time &		Obj. value & Time \\ \hline
1 & 9330	&	0.11	&	8315	&	0.03	&	\bf 9425	&	16.86	&	\bf 9425	&	3.25 \\
2 & 9080	&	0.08	&	8060	&	0.03	&	\bf 9200	&	20.92	&	\bf 9200	&	2.86 \\
3 & 9330	&	0.11	&	8315	&	0.03	&	\bf 9425	&	15.07	&	\bf 9425	&	3.11 \\
4 & \bf 9130	&	0.08	&	7785	&	0.03	&	\bf 9130	&	22.61	&	\bf 9130	&	1.53 \\
5 & \bf 9420	&	0.10	&	7500	&	0.02	&	\bf 9420	&	24.90	&	\bf 9420	&	2.76 \\
6 & 10245	&	0.12	&	8380	&	0.02	&	\bf 10320	&	18.95	&	\bf 10320	&	1.08 \\
7 & 8675	&	0.06	&	7565	&	0.02	&	8675	&	21.57	&	\bf 8700	&	3.15 \\
8 & \bf 9260	&	0.09	&	7575	&	0.03	&	\bf 9260	&	29.56	&	\bf 9260	&	1.40 \\
9 & \bf 9885	&	0.10 &	8180	&	0.02	&	\bf 9885	&	20.86	&	\bf 9885	&	1.23 \\
10 & \bf 9370	&	0.09	&	7425	&	0.03	&	9350	&	15.14	&	\bf 9370	&	3.29 \\
\hline
    \end{tabular}
    \caption{The worst optimal value found by the local search algorithm (with the first-improvement and the best-improvement policy), the genetic algorithm and the memetic algorithm proposed in this paper, for the $20 \times 20$ instances from \cite{Xie:UpperBoundMinimal:2017}. The highest obtained value for each instance is marked in bold. The time (in seconds) required to find the solution by each method is also reported.}\label{tab:exp:memetics}
    \end{table}

\pgfplotstableread[col sep=space]{
instance Algorithm avg_obj min_obj max_obj avg_time min_time max_time draw
1	2		9119	8560	9330		0.108861923217773	0.019597053527832	0.167963266372681 1
1	3		7365	6795	8315		0.0290294647216797	0.0184903144836426	0.0443811416625977 2			
1	4		9250	9145	9425		16.85949010849	11.6849849224091	22.6391575336456 3		 	
1	5		9421	9405	9425		3.25462827682495	0.649504661560059	7.93840098381043 4

2	2		8990	8900	9080		0.0839365482330322	0.0619988441467285	0.101388454437256 5		
2	3		7268	6700	8060		0.028938627243042	0.0192856788635254	0.0401389598846436 6		
2	4		9168	9080	9200		20.9158733844757	15.1739628314972	25.9510290622711 7			
2	5		9196	9180	9200		2.86339631080627	1.23249316215515	4.95468878746033 8

3	2		9119	8560	9330		0.101899433135986	0.0191242694854736	0.133016109466553	 9		
3	3		7365	6795	8315		0.0268145084381104	0.0163931846618652	0.0416445732116699		10	
3	4		9250	9145	9425		15.0741301059723	8.48294973373413	19.1919207572937			11
3	5		9421	9405	9425		3.1095057964325	0.641295194625855	7.8110044002533			12

4	2		8777	8435	9130		0.0755904674530029	0.0468418598175049	0.0966901779174805		13	
4	3		7252	6845	7785		0.0291649341583252	0.0236268043518066	0.0322070121765137			14
4	4		9130	9130	9130		22.6131054878235	15.7554161548615	26.8890945911407			15
4	5		9130	9130	9130		1.52703108787537	0.736182689666748	3.093177318573			16

5	2		8943	7430	9420		0.0967044353485107	0.0283973217010498	0.167765140533447			17
5	3		6827	6150	7500		0.0232756614685059	0.0170390605926514	0.03243088722229			18
5	4		9351	9190	9420		24.8959182262421	18.6146204471588	33.7956140041351			19
5	5		9420	9420	9420		2.76172499656677	0.611107110977173	6.17563247680664			20

6	2		10291	10245	10320		0.120388221740723	0.109599351882935	0.136763095855713			21
6	3		7822	6925	8380		0.0234377384185791	0.016284704208374	0.0285882949829102			22
6	4		10320	10320	10320		18.948573923111	15.0167415142059	21.7831890583038			23
6	5		10320	10320	10320		1.08108611106873	0.76025915145874	1.79930543899536		24

7	2		8271	7715	8675		0.06091628074646	0.0212035179138184	0.0971536636352539		25	
7	3		7131	6350	7565		0.0211897850036621	0.017838716506958	0.0240697860717773		26	
7	4		8675	8675	8675		21.5703828334808	16.0411522388458	24.0196805000305		27	
7	5		8700	8700	8700		3.15477337837219	0.795957803726196	5.09900569915772		28

8	2		8662	7550	9260		0.0910737991333008	0.0171384811401367	0.165892839431763			29
8	3		6649	5745	7575		0.026797342300415	0.0159423351287842	0.0401830673217773			30
8	4		9260	9260	9260		29.5595542430878	16.1216428279877	43.2323489189148			31
8	5		9260	9260	9260		1.40193791389465	0.880597591400147	1.79532957077026			32

9	2		9823	9575	9885		0.0952373504638672	0.0662455558776856	0.14861011505127			33
9	3		7820	7510	8180		0.0245130062103272	0.0170173645019531	0.0318007469177246			34
9	4		9885	9885	9885		20.8644217014313	15.9952385425568	29.9151105880737			35
9	5		9885	9885	9885		1.22529759407043	0.64405083656311	2.26242089271545			36

10	2		9011	7815	9370		0.0948210716247559	0.0291793346405029	0.133568525314331			37
10	3		6720	6195	7425		0.0330136775970459	0.017686128616333	0.0645005702972412			38
10	4		9276	9175	9350		15.1416352748871	10.234747171402	18.9932994842529			39
10	5		9370	9370	9370		5.4245876789093	3.28995943069458	9.7558376789093			40

}\datatable

\pgfplotstablecreatecol[
    create col/expr={\thisrow{avg_obj} - \thisrow{min_obj}}
]{obj_err_low}{\datatable}

\pgfplotstablecreatecol[
    create col/expr={\thisrow{max_obj} - \thisrow{avg_obj}}
]{obj_err_high}{\datatable}

\pgfplotstablecreatecol[
    create col/expr={\thisrow{avg_time} - \thisrow{min_time}}
]{time_err_low}{\datatable}

\pgfplotstablecreatecol[
    create col/expr={\thisrow{max_time} - \thisrow{avg_time}}
]{time_err_high}{\datatable}

\begin{figure}[h]
\centering
\begin{tikzpicture}
\begin{axis}[
    xtick={1,...,20},
    xticklabels={2,3,4,5,2,3,4,5,2,3,4,5,2,3,4,5,2,3,4,5},
    grid=both,
    grid style=dashed,
    width=12.5cm,
    height=5cm,
    scaled y ticks=false
]

\addplot+[
    only marks,
    error bars/.cd,
        y dir=both,
        y explicit,
]
table[
    x=draw,
    y=avg_obj,
    y error minus=obj_err_low,
    y error plus=obj_err_high,
    restrict expr to domain={\thisrow{instance}}{1:1},
]{\datatable};

\addplot+[
    only marks,
    error bars/.cd,
        y dir=both,
        y explicit,
]
table[
    x=draw,
    y=avg_obj,
    y error minus=obj_err_low,
    y error plus=obj_err_high,
    restrict expr to domain={\thisrow{instance}}{2:2},
]{\datatable};

\addplot+[
    only marks,
    error bars/.cd,
        y dir=both,
        y explicit,
]
table[
    x=draw,
    y=avg_obj,
    y error minus=obj_err_low,
    y error plus=obj_err_high,
    restrict expr to domain={\thisrow{instance}}{3:3},
]{\datatable};

\addplot+[
    only marks,
    error bars/.cd,
        y dir=both,
        y explicit,
]
table[
    x=draw,
    y=avg_obj,
    y error minus=obj_err_low,
    y error plus=obj_err_high,
    restrict expr to domain={\thisrow{instance}}{4:4},
]{\datatable};

\addplot+[
    only marks,
    error bars/.cd,
        y dir=both,
        y explicit,
]
table[
    x=draw,
    y=avg_obj,
    y error minus=obj_err_low,
    y error plus=obj_err_high,
    restrict expr to domain={\thisrow{instance}}{5:5},
]{\datatable};
\end{axis}
\end{tikzpicture}
\begin{tikzpicture}
\begin{axis}[
    xtick={21,...,40},
    xticklabels={2,3,4,5,2,3,4,5,2,3,4,5,2,3,4,5,2,3,4,5},
    grid=both,
    grid style=dashed,
     width=12.5cm,
    height=5cm,
    scaled y ticks=false
]
\addplot+[
    only marks,
    error bars/.cd,
        y dir=both,
        y explicit,
]
table[
    x = draw,
    y=avg_obj,
    y error minus=obj_err_low,
    y error plus=obj_err_high,
    restrict expr to domain={\thisrow{instance}}{6:6},
]{\datatable};

\addplot+[
    only marks,
    error bars/.cd,
        y dir=both,
        y explicit,
]
table[
    x = draw,
    y=avg_obj,
    y error minus=obj_err_low,
    y error plus=obj_err_high,
    restrict expr to domain={\thisrow{instance}}{7:7},
]{\datatable};

\addplot+[
    only marks,
    error bars/.cd,
        y dir=both,
        y explicit,
]
table[
    x = draw,
    y=avg_obj,
    y error minus=obj_err_low,
    y error plus=obj_err_high,
    restrict expr to domain={\thisrow{instance}}{8:8},
]{\datatable};

\addplot+[
    only marks,
    error bars/.cd,
        y dir=both,
        y explicit,
]
table[
    x = draw,
    y=avg_obj,
    y error minus=obj_err_low,
    y error plus=obj_err_high,
    restrict expr to domain={\thisrow{instance}}{9:9},
]{\datatable};

\addplot+[
    only marks,
    error bars/.cd,
        y dir=both,
        y explicit,
]
table[
    x = draw,
    y=avg_obj,
    y error minus=obj_err_low,
    y error plus=obj_err_high,
    restrict expr to domain={\thisrow{instance}}{10:10},
]{\datatable};
\end{axis}
\end{tikzpicture}
\caption{An interval plot of the maximal, minimal and average objective values attained by the four algorithms presented in this paper over $5$ runs for instances $1$-$5$ (above) and $6$-$10$ (below). The labels indicate the number of the algorithm used: local search (\ref{alg:localSearch:first} and \ref{alg:localSearch:best}), the genetic algorithm (\ref{alg:geneticAlg}) and the memetic algorithm (\ref{alg:memeticAlg}).}\label{fig:exp:intplot:obj}
\end{figure}
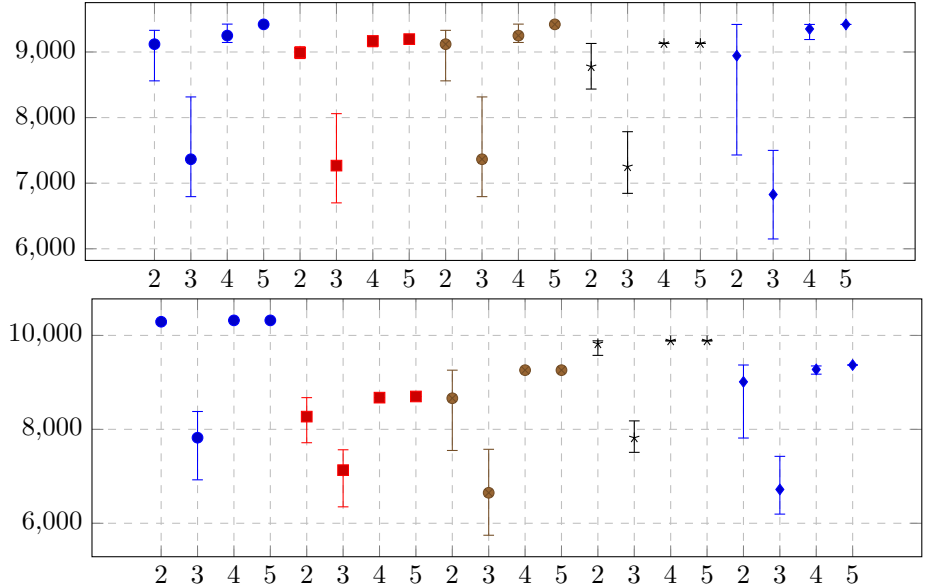

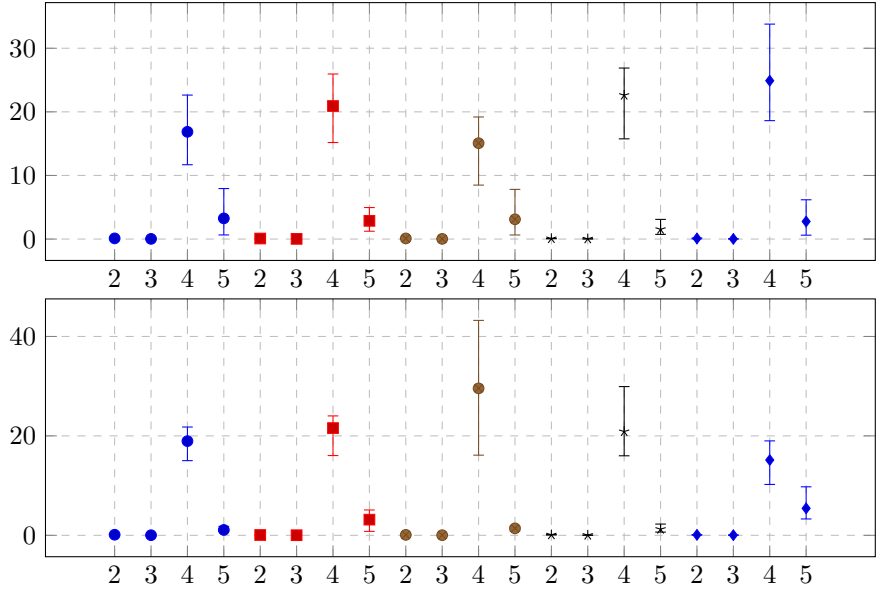
\begin{figure}[h]
\centering
\begin{tikzpicture}
\begin{axis}[
    xtick={1,...,20},
    xticklabels={2,3,4,5,2,3,4,5,2,3,4,5,2,3,4,5,2,3,4,5},
    grid=both,
    grid style=dashed,
    width=12.5cm,
    height=5cm,
    scaled y ticks=false
]

\addplot+[
    only marks,
    error bars/.cd,
        y dir=both,
        y explicit,
]
table[
    x=draw,
    y=avg_time,
    y error minus=time_err_low,
    y error plus=time_err_high,
    restrict expr to domain={\thisrow{instance}}{1:1},
]{\datatable};

\addplot+[
    only marks,
    error bars/.cd,
        y dir=both,
        y explicit,
]
table[
    x=draw,
    y=avg_time,
    y error minus=time_err_low,
    y error plus=time_err_high,
    restrict expr to domain={\thisrow{instance}}{2:2},
]{\datatable};

\addplot+[
    only marks,
    error bars/.cd,
        y dir=both,
        y explicit,
]
table[
    x=draw,
    y=avg_time,
    y error minus=time_err_low,
    y error plus=time_err_high,
    restrict expr to domain={\thisrow{instance}}{3:3},
]{\datatable};

\addplot+[
    only marks,
    error bars/.cd,
        y dir=both,
        y explicit,
]
table[
    x=draw,
    y=avg_time,
    y error minus=time_err_low,
    y error plus=time_err_high,
    restrict expr to domain={\thisrow{instance}}{4:4},
]{\datatable};

\addplot+[
    only marks,
    error bars/.cd,
        y dir=both,
        y explicit,
]
table[
    x=draw,
    y=avg_time,
    y error minus=time_err_low,
    y error plus=time_err_high,
    restrict expr to domain={\thisrow{instance}}{5:5},
]{\datatable};
\end{axis}
\end{tikzpicture}
\begin{tikzpicture}
\begin{axis}[
    xtick={21,...,40},
    xticklabels={2,3,4,5,2,3,4,5,2,3,4,5,2,3,4,5,2,3,4,5},
    grid=both,
    grid style=dashed,
     width=12.5cm,
    height=5cm,
    scaled y ticks=false
]
\addplot+[
    only marks,
    error bars/.cd,
        y dir=both,
        y explicit,
]
table[
    x = draw,
    y=avg_time,
    y error minus=time_err_low,
    y error plus=time_err_high,
    restrict expr to domain={\thisrow{instance}}{6:6},
]{\datatable};

\addplot+[
    only marks,
    error bars/.cd,
        y dir=both,
        y explicit,
]
table[
    x = draw,
    y=avg_time,
    y error minus=time_err_low,
    y error plus=time_err_high,
    restrict expr to domain={\thisrow{instance}}{7:7},
]{\datatable};

\addplot+[
    only marks,
    error bars/.cd,
        y dir=both,
        y explicit,
]
table[
    x = draw,
    y=avg_time,
    y error minus=time_err_low,
    y error plus=time_err_high,
    restrict expr to domain={\thisrow{instance}}{8:8},
]{\datatable};

\addplot+[
    only marks,
    error bars/.cd,
        y dir=both,
        y explicit,
]
table[
    x = draw,
    y=avg_time,
    y error minus=time_err_low,
    y error plus=time_err_high,
    restrict expr to domain={\thisrow{instance}}{9:9},
]{\datatable};

\addplot+[
    only marks,
    error bars/.cd,
        y dir=both,
        y explicit,
]
table[
    x = draw,
    y=avg_time,
    y error minus=time_err_low,
    y error plus=time_err_high,
    restrict expr to domain={\thisrow{instance}}{10:10},
]{\datatable};
\end{axis}
\end{tikzpicture}
\caption{An interval plot of the maximal, minimal and average running time of the four algorithms presented in this paper over $5$ runs for instances $1$-$5$ (above) and $6$-$10$ (below). The labels indicate the number of the algorithm used: local search (\ref{alg:localSearch:first} and \ref{alg:localSearch:best}), the genetic algorithm (\ref{alg:geneticAlg}) and the memetic algorithm (\ref{alg:memeticAlg}).}\label{fig:exp:intplot:time}
\end{figure}

    \subsubsection{Comparison against the published methods}
    Table~\ref{tab:exp:all} presents the results of the computational experiment for the three previously published methods for approximating the worst finite optimal value of interval transportation problems and for the memetic algorithm proposed in this paper. The values for \cite{Xie:UpperBoundMinimal:2017} and \cite{Cerulli:BestWorstValues:2017} are based on the results previously published in the literature. Running times (in seconds) are reported for the mixed-integer linear programming model by \cite{Garajova:IntervalTransportationProblem:2023} solved with Gurobi (with a time limit of $1$ minute) and for the memetic algorithm implemented by the authors. 
    
    Note that the solutions found by the integer programming model are not necessarily optimal, since the solver was not able to solve the program to optimality within the time limit. In fact, in instances $9$ and $10$ the memetic algorithm was able to find a scenario with a higher optimal value than the value attained by Gurobi. For instances $2$ and $10$, the objective value of $9200$ and $9370$ returned by the memetic algorithm is even higher than the highest bound that was previously found by any of the competing algorithms. In all of the instances, the memetic algorithm was able to compute the tightest approximation of the worst optimal value in under $4$ seconds of running time.

    \begin{table}[b]
    \begin{tabular}{cKKKMKM}
    \hline
    & \cite{Xie:UpperBoundMinimal:2017} & \cite{Cerulli:BestWorstValues:2017} & \cite{Garajova:IntervalTransportationProblem:2023} & Time & Memetic alg. (Alg. \ref{alg:memeticAlg}) & Time \\ \hline
1	&	9405	&	9405	&	\bf 9425	&	8	&	\bf 9425	&	3.25 \\
2	&	9015	&	9140	&	\bf 9200	&	59	&	\bf 9200	&	2.86 \\
3	&	9335	&	9405	&	\bf 9425	&	8	&	\bf 9425	&	3.11 \\
4	&	8930	&	\bf 9130	&	\bf 9130	&	12	&	\bf 9130	&	1.53 \\
5	&	9275	&	\bf 9420	&	\bf 9420	&	1	&	\bf 9420	&	2.76 \\
6	&	10220	&	\bf 10320	&	\bf 10320	&	8	&	\bf 10320	&	1.08 \\
7	&	8685	&	8630	&	\bf 8700	&	5	&	\bf 8700	&	3.15 \\
8	&	\bf 9260	&	\bf 9260	&	\bf 9260	&	0	&	\bf 9260	&	1.40 \\
9	&	\bf 9885	&	\bf 9885	&	9685	&	2	&	\bf 9885	&	1.23 \\
10	&	9225	&	9220	&	9200	&	59	&	\bf 9370	&	3.29 \\
\hline
    \end{tabular}
    \caption{The worst optimal value found by the four compared algorithms for the $20 \times 20$ instances from \cite{Xie:UpperBoundMinimal:2017}. The highest obtained value for each instance is marked in bold. The time (in seconds) required to find the solution by the MILP model and the memetic algorithm is also reported in the respective columns. The results for the MILP model \citep{Garajova:IntervalTransportationProblem:2023} are recomputed using the newer version of the Gurobi solver (differences are for instances 2 and 10 (better solutions this time) and 9 (a worse solution this time).}\label{tab:exp:all}
    \end{table}

    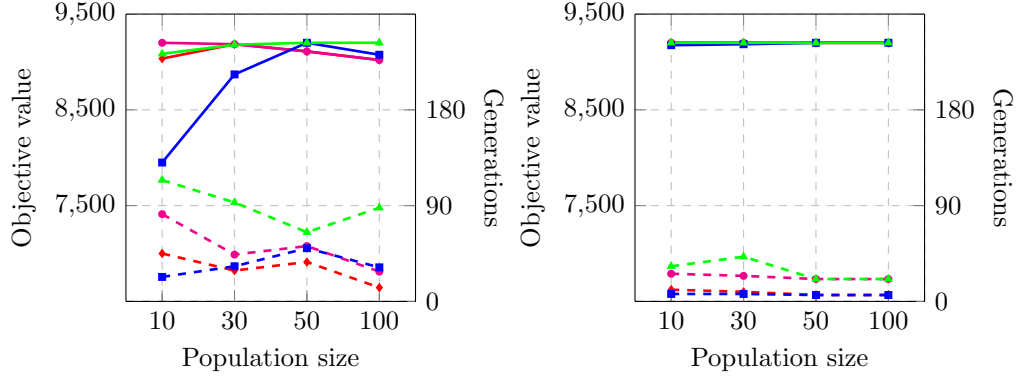
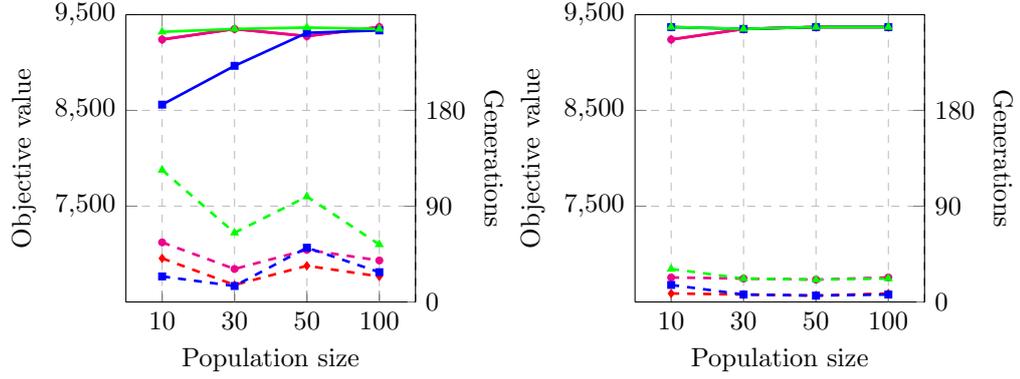
\begin{figure}[h!]\centering
    \begin{subfigure}{\textwidth}
    \begin{tikzpicture}

\begin{axis}[%
grid=major,grid style={dashed},
width=1.5in,
height=1.5in,
scale only axis,
xmin=0.5,
xmax=4.5,
separate axis lines,
every outer y axis line/.append style={black},
every y tick label/.append style={font=\color{black}},
ymin=6500,
ymax=9500,
ytick       ={    7500,  8500, 9500},
mark size=1pt,
ylabel      ={Objective value},
xtick       ={1,2,3,4},
xticklabels ={10,30,50,100},
xlabel={Population size},
line width=1pt,
axis line style={line width=0.4pt},
]
\addplot [
color=red,
mark=diamond*,
solid,
]
table[row sep=crcr]{
1	9035	\\
2	9185	\\
3	9110	\\
4	9020	\\
}; \label{plot:2:linshift:5}
\addplot [
color=magenta,
mark=*,
solid,
]
table[row sep=crcr]{
1	9200	\\
2	9185	\\
3	9110	\\
4	9020	\\
}; \label{plot:2:linshift:20}
\addplot [
color=blue,
mark=square*,
solid,
]
table[row sep=crcr]{
1	7950	\\
2	8870	\\
3	9200	\\
4	9075	\\
}; \label{plot:2:tournament:5}
\addplot [
color=green,
mark=triangle*,
solid,
]
table[row sep=crcr]{
1	9085	\\
2	9180	\\
3	9200	\\
4	9200	\\
}; \label{plot:2:tournament:20}
\end{axis}

\begin{axis}[%
width=1.5in,
height=1.5in,
scale only axis,
dashed,
axis line style={line width=0pt},
xmin=0.5,
xmax=4.5,
every outer y axis line/.append style={black},
every y tick label/.append style={font=\color{black}},
ymin=0,
ymax=270,
line width=1pt,
mark options={solid},
mark size=1pt,
ytick={   0, 90,  180},
axis x line*=bottom,
axis y line*=right,
xtick={},
xticklabels={},
ylabel={Generations},
ylabel style={rotate=180},
]
\addplot [
color=red,
mark=diamond*,
forget plot
]
table[row sep=crcr]{
1	45	\\
2	29	\\
3	37	\\
4	13	\\
};

\addplot [
color=magenta,
mark=*,
forget plot
]
table[row sep=crcr]{
1	82	\\
2	44	\\
3	52	\\
4	28	\\
};

\addplot [
color=blue,
mark=square*,
forget plot
]
table[row sep=crcr]{
1	23	\\
2	33	\\
3	50	\\
4	32	\\
}; 

\addplot [
color=green,
mark=triangle*,
forget plot
]
table[row sep=crcr]{
1	114	\\
2	93	\\
3	65	\\
4	88	\\
}; 
\end{axis}
\end{tikzpicture}%
\begin{tikzpicture}

\begin{axis}[%
grid=major,grid style={dashed},
width=1.5in,
height=1.5in,
scale only axis,
xmin=0.5,
xmax=4.5,
separate axis lines,
every outer y axis line/.append style={black},
every y tick label/.append style={font=\color{black}},
ymin=6500,
ymax=9500,
ytick       ={    7500,  8500, 9500},
mark size=1pt,
ylabel      ={Objective value},
xtick       ={1,2,3,4},
xticklabels ={10,30,50,100},
xlabel={Population size},
line width=1pt,
axis line style={line width=0.4pt},
]
\addplot [
color=red,
mark=diamond*,
solid,
]
table[row sep=crcr]{
1	9200	\\
2	9200	\\
3	9200	\\
4	9200	\\
}; 
\addplot [
color=magenta,
mark=*,
solid,
]
table[row sep=crcr]{
1	9200	\\
2	9200	\\
3	9200	\\
4	9200	\\
}; 
\addplot [
color=blue,
mark=square*,
solid,
]
table[row sep=crcr]{
1	9175	\\
2	9185	\\
3	9200	\\
4	9200	\\
}; 
\addplot [
color=green,
mark=triangle*,
solid,
]
table[row sep=crcr]{
1	9200	\\
2	9200	\\
3	9200	\\
4	9200	\\
}; 

\end{axis}

\begin{axis}[%
width=1.5in,
height=1.5in,
scale only axis,
dashed,
axis line style={line width=0pt},
xmin=0.5,
xmax=4.5,
every outer y axis line/.append style={black},
every y tick label/.append style={font=\color{black}},
ymin=0,
ymax=270,
line width=1pt,
mark options={solid},
mark size=1pt,
ytick={   0, 90,  180},
ylabel={Generations},
axis x line*=bottom,
axis y line*=right,
xtick={},
xticklabels={},
ylabel style={rotate=180}
]
\addplot [
color=red,
mark=diamond*,
forget plot
]
table[row sep=crcr]{
1	11	\\
2	9	\\
3	6	\\
4	6	\\
}; 

\addplot [
color=magenta,
mark=*,
forget plot
]
table[row sep=crcr]{
1	26	\\
2	24	\\
3	21	\\
4	21	\\
};

\addplot [
color=blue,
mark=square*,
forget plot
]
table[row sep=crcr]{
1	7	\\
2	7	\\
3	6	\\
4	6	\\
}; 

\addplot [
color=green,
mark=triangle*,
forget plot
]
table[row sep=crcr]{
1	33	\\
2	42	\\
3	21	\\
4	21	\\
}; 
\end{axis}
\end{tikzpicture}%
\caption{Instance 2}
\end{subfigure}
\begin{subfigure}{\textwidth}
\begin{tikzpicture}

\begin{axis}[%
grid=major,grid style={dashed},
width=1.5in,
height=1.5in,
scale only axis,
xmin=0.5,
xmax=4.5,
separate axis lines,
every outer y axis line/.append style={black},
every y tick label/.append style={font=\color{black}},
ymin=6500,
ymax=9500,
ytick       ={    7500,  8500, 9500},
mark size=1pt,
ylabel      ={Objective value},
xtick       ={1,2,3,4},
xticklabels ={10,30,50,100},
xlabel={Population size},
line width=1pt,
axis line style={line width=0.4pt},
]
\addplot [
color=red,
mark=diamond*,
solid,
]
table[row sep=crcr]{
1	9240	\\
2	9350	\\
3	9275	\\
4	9370	\\
}; 
\addplot [
color=magenta,
mark=*,
solid,
]
table[row sep=crcr]{
1	9240	\\
2	9350	\\
3	9275	\\
4	9370	\\
}; 
\addplot [
color=blue,
mark=square*,
solid,
]
table[row sep=crcr]{
1	8560	\\
2	8965	\\
3	9310	\\
4	9335	\\
}; 
\addplot [
color=green,
mark=triangle*,
solid,
]
table[row sep=crcr]{
1	9320	\\
2	9350	\\
3	9365	\\
4	9350	\\
}; 
\end{axis}

\begin{axis}[%
width=1.5in,
height=1.5in,
scale only axis,
dashed,
axis line style={line width=0pt},
xmin=0.5,
xmax=4.5,
every outer y axis line/.append style={black},
every y tick label/.append style={font=\color{black}},
ymin=0,
ymax=270,
line width=1pt,
mark options={solid},
mark size=1pt,
ytick={   0, 90,  180},
ylabel={Generations},
axis x line*=bottom,
axis y line*=right,
xtick={},
xticklabels={},
ylabel style={rotate=180}
]
\addplot [
color=red,
mark=diamond*,
forget plot
]
table[row sep=crcr]{
1	41	\\
2	16	\\
3	34	\\
4	24	\\
}; \label{areaplot}

\addplot [
color=magenta,
mark=*,
forget plot
]
table[row sep=crcr]{
1	56	\\
2	31	\\
3	49	\\
4	39	\\
}; \label{areaplot}

\addplot [
color=blue,
mark=square*,
forget plot
]
table[row sep=crcr]{
1	24	\\
2	15	\\
3	51	\\
4	28	\\
}; \label{areaplot}

\addplot [
color=green,
mark=triangle*,
forget plot
]
table[row sep=crcr]{
1	124	\\
2	65	\\
3	99	\\
4	54	\\
}; \label{areaplot}
\end{axis}
\end{tikzpicture}%
\begin{tikzpicture}

\begin{axis}[%
grid=major,grid style={dashed},
width=1.5in,
height=1.5in,
scale only axis,
xmin=0.5,
xmax=4.5,
separate axis lines,
every outer y axis line/.append style={black},
every y tick label/.append style={font=\color{black}},
ymin=6500,
ymax=9500,
ytick       ={    7500,  8500, 9500},
mark size=1pt,
ylabel      ={Objective value},
xtick       ={1,2,3,4},
xticklabels ={10,30,50,100},
xlabel={Population size},
line width=1pt,
axis line style={line width=0.4pt},
]
\addplot [
color=red,
mark=diamond*,
solid,
]
table[row sep=crcr]{
1	9240	\\
2	9350	\\
3	9370	\\
4	9370	\\
}; 
\addplot [
color=magenta,
mark=*,
solid,
]
table[row sep=crcr]{
1	9240	\\
2	9350	\\
3	9370	\\
4	9370	\\
}; 
\addplot [
color=blue,
mark=square*,
solid,
]
table[row sep=crcr]{
1	9370	\\
2	9350	\\
3	9370	\\
4	9370	\\
}; 
\addplot [
color=green,
mark=triangle*,
solid,
]
table[row sep=crcr]{
1	9370	\\
2	9350	\\
3	9370	\\
4	9370	\\
}; 

\label{displacementplot}
\end{axis}

\begin{axis}[%
width=1.5in,
height=1.5in,
scale only axis,
dashed,
axis line style={line width=0pt},
xmin=0.5,
xmax=4.5,
every outer y axis line/.append style={black},
every y tick label/.append style={font=\color{black}},
ymin=0,
ymax=270,
line width=1pt,
mark options={solid},
mark size=1pt,
ytick={   0, 90,  180},
ylabel={Generations},
axis x line*=bottom,
axis y line*=right,
xtick={},
xticklabels={},
ylabel style={rotate=180}
]
\addplot [
color=red,
mark=diamond*,
forget plot
]
table[row sep=crcr]{
1	8	\\
2	7	\\
3	6	\\
4	8	\\
}; \label{areaplot}

\addplot [
color=magenta,
mark=*,
forget plot
]
table[row sep=crcr]{
1	23	\\
2	22	\\
3	21	\\
4	23	\\
}; \label{areaplot}

\addplot [
color=blue,
mark=square*,
forget plot
]
table[row sep=crcr]{
1	16	\\
2	7	\\
3	6	\\
4	7	\\
}; \label{areaplot}

\addplot [
color=green,
mark=triangle*,
forget plot
]
table[row sep=crcr]{
1	31	\\
2	22	\\
3	21	\\
4	22	\\
}; \label{areaplot}
\end{axis}
\end{tikzpicture}%
\caption{Instance 10}
\end{subfigure}
    \caption{The objective value and the number of generations of the genetic algorithm (left) and the memetic algorithm (right) on two instances for population size $N_{\text{pop}} \in \{10,30,50,100\}$. Different settings of parameters are used: fitness proportionate selection with strategy $\pi_S^3$ and $t_{\text{GA}} = 5$ (\ref{plot:2:linshift:5}), fitness proportionate selection with strategy $\pi_S^3$ and $t_{\text{GA}} = 20$ (\ref{plot:2:linshift:20}), tournament selection with $t_{\text{GA}} = 5$ (\ref{plot:2:tournament:5}) and tournament selection with $t_{\text{GA}} = 20$ (\ref{plot:2:tournament:20}).}\label{fig:exp:gen}
    \end{figure}

  \subsubsection{Properties of the genetic algorithm}

    Figure~\ref{fig:exp:gen} illustrates the objective value obtained by the genetic and the memetic algorithm and the number of generations processed to compute the value on two instances (instance $2$ and instance $10$, respectively) for different initial population sizes and under various settings of the parameters. Namely, population size $N_\text{pop} \in \{10, 30, 50, 100\}$ is tested with either tournament selection or fitness proportionate selection using strategy $\pi_S^3$ (see Section~\ref{ssec:genetic:fitness} for details) and with the termination condition determined by the number of nonimproved iterations $t_\text{GA} \in \{5, 20\}$.
    
    We can see that the behavior of the genetic algorithm varies more with the population size (especially with the tournament selection strategy), while the memetic algorithm exhibits a more stable behavior. However, both of the population-based algorithms are able to produce high-quality solutions even with a small population size. In fact, when using tournament selection, the memetic algorithm was able to find the highest objective value of both instances even with $N_\text{pop} = 10$. 

    Figure~\ref{fig:exp:gen:progress} shows the progress of the genetic and the memetic algorithm when solving instance $2$ under two different settings: fitness proportionate selection using $\pi_S^3$ with $N_\text{pop} = 10$ and $t_\text{GA} = 20$ and tournament selection with $N_\text{pop} = 50$ and $t_\text{GA} = 5$. For this instance, all of the tested algorithms found the same highest value of $9200$. 
    
    The longest time to return the solution was observed for the genetic algorithm with $N_\text{pop} = 50$, which processed $50$ generations in about $27$ seconds. All of the remaining algorithms were able to compute the solution in under $8$ seconds. The genetic algorithm with $N_\text{pop} = 10$ processed $80$ generations, but since the population in each generation was $5$ times smaller, the iterations were much faster.

    The memetic algorithm was able to find the highest value quickly and spent most of the time waiting on the terminating condition for $t_\text{GA}$ iterations with this value. In this case, the number of generations was lower than for the genetic algorithm ($6$ and $26$), but the iterations were more demanding due to the inclusion of local search.

\begin{figure}[]
\begin{tikzpicture}[scale=0.9]

\begin{axis}[%
grid=major,grid style={dashed},
width=5in,
height=2in,
scale only axis,
xmin=0,
xmax=27.5,
separate axis lines,
every outer y axis line/.append style={black},
every y tick label/.append style={font=\color{black}},
ymin=6800,
ymax=9500,
ytick       ={    7500,  8500, 9500},
mark size=1pt,
ylabel      ={Objective value},
xtick       ={0,1,...,27},
xticklabels ={0,1,...,27},
xlabel={Time},
line width=1pt,
axis line style={line width=0.4pt},
]
\addplot [
color=red,
mark=diamond*,
solid,
]
table[row sep=crcr]{
0 6800 \\
0.092522	7010	\\
0.186437	7040	\\
0.279204	7575	\\
0.371959	7670	\\
0.464543	7820	\\
0.556570	7840	\\
0.649044	7930	\\
0.741353	7930	\\
0.833477	7950	\\
0.926039	8225	\\
1.019101	8225	\\
1.130646	8420	\\
1.232148	8420	\\
1.330330	8420	\\
1.440007	8460	\\
1.542689	8460	\\
1.654523	8500	\\
1.757582	8510	\\
1.858459	8530	\\
1.956721	8540	\\
2.057988	8595	\\
2.154844	8595	\\
2.252267	8605	\\
2.346057	8605	\\
2.490473	8605	\\
2.620367	8605	\\
2.730686	8630	\\
2.814110	8630	\\
2.904621	8660	\\
2.990024	8700	\\
3.074703	8700	\\
3.165223	8850	\\
3.259175	8850	\\
3.341642	8850	\\
3.432270	8855	\\
3.516654	8855	\\
3.616023	8885	\\
3.708613	8885	\\
3.801262	8945	\\
3.884998	9035	\\
3.968616	9035	\\
4.050951	9035	\\
4.132879	9035	\\
4.215125	9035	\\
4.303964	9035	\\
4.399253	9035	\\
4.489393	9035	\\
4.577390	9035	\\
4.670621	9090	\\
4.756597	9090	\\
4.846478	9090	\\
4.937992	9110	\\
5.032342	9110	\\
5.126170	9110	\\
5.226723	9110	\\
5.323079	9150	\\
5.425366	9195	\\
5.524732	9195	\\
5.628304	9195	\\
5.727762	9195	\\
5.834203	9195	\\
5.936046	9200	\\
6.035460	9200	\\
6.130842	9200	\\
6.228178	9200	\\
6.325043	9200	\\
6.422900	9200	\\
6.521228	9200	\\
6.620746	9200	\\
6.717820	9200	\\
6.813961	9200	\\
6.908972	9200	\\
7.003839	9200	\\
7.100595	9200	\\
7.196354	9200	\\
7.295783	9200	\\
7.392859	9200	\\
7.495229	9200	\\
7.598429	9200	\\
7.699309	9200	\\
7.794709	9200	\\
7.900644	9200	\\
}; \label{plot:progress:genetics:10:linear:20}
\addplot [
color=magenta,
mark=*,
solid,
]
table[row sep=crcr]{
0 6800 \\
0.295779	8940	\\
0.531831	9110	\\
0.716162	9185	\\
0.886475	9185	\\
1.055772	9185	\\
1.264498	9200	\\
1.468291	9200	\\
1.639659	9200	\\
1.845943	9200	\\
2.062498	9200	\\
2.278285	9200	\\
2.475368	9200	\\
2.649678	9200	\\
2.865865	9200	\\
3.055664	9200	\\
3.258637	9200	\\
3.414000	9200	\\
3.611689	9200	\\
3.795964	9200	\\
3.957141	9200	\\
4.156111	9200	\\
4.358616	9200	\\
4.545477	9200	\\
4.734426	9200	\\
4.929767	9200	\\
5.133550	9200	\\
}; \label{plot:progress:memetics:10:linear:20}
\addplot [
color=blue,
mark=square*,
solid,
]
table[row sep=crcr]{
0 6800 \\
0.512120	7610	\\
1.021178	7835	\\
1.511312	7835	\\
2.032502	8095	\\
2.580220	8120	\\
3.167374	8130	\\
3.690908	8230	\\
4.195767	8535	\\
4.713091	8535	\\
5.198233	8555	\\
5.713408	8555	\\
6.248279	8555	\\
6.743797	8555	\\
7.266685	8710	\\
7.760638	8710	\\
8.254118	8845	\\
8.786330	8845	\\
9.302532	8845	\\
9.820230	8845	\\
10.350824	8845	\\
10.919719	8860	\\
11.480425	8910	\\
12.024157	8935	\\
12.571874	8935	\\
13.108973	9010	\\
13.645392	9010	\\
14.203255	9010	\\
14.793345	9010	\\
15.350878	9030	\\
15.932735	9055	\\
16.472825	9055	\\
17.016390	9100	\\
17.552307	9100	\\
18.065207	9100	\\
18.608982	9100	\\
19.197684	9100	\\
19.757303	9150	\\
20.304439	9150	\\
20.833913	9150	\\
21.349021	9150	\\
21.894909	9150	\\
22.439456	9155	\\
22.938985	9175	\\
23.474202	9175	\\
24.016636	9200	\\
24.585746	9200	\\
25.151430	9200	\\
25.719399	9200	\\
26.265036	9200	\\
26.807969	9200	\\
}; \label{plot:progress:genetics:50:tournament:5}
\addplot [
color=green,
mark=triangle*,
mark size=3pt,
solid,
]
table[row sep=crcr]{
0 6800 \\
1.345562	9200	\\
2.360399	9200	\\
3.391848	9200	\\
4.227045	9200	\\
4.980540	9200	\\
5.669363	9200	\\
}; \label{plot:progress:memetics:50:tournament:5}

\end{axis}

\end{tikzpicture}%
\caption{Instance 2. \ref{plot:progress:genetics:10:linear:20} and \ref{plot:progress:memetics:10:linear:20}: genetic and memetic with $N_\text{pop} = 10$, strategy $\pi^3_S$, $t_\text{GA} = 20$. \ref{plot:progress:genetics:50:tournament:5} and \ref{plot:progress:memetics:50:tournament:5}: genetic and memetic with $N_\text{pop} = 50$, tournament selection strategy, $t_\text{GA} = 5$.\\
All algorithms found the same solution ($9200$). Numbers of generations: \ref{plot:progress:genetics:10:linear:20}: 80, but the cheapest iterations; \ref{plot:progress:memetics:10:linear:20} 26, but more complex generations (due to local search) makes it only 2 sec faster than genetic; \ref{plot:progress:genetics:50:tournament:5}: 50, the bigger population makes the generations more demanding; \ref{plot:progress:memetics:50:tournament:5}: 6, the local search is demanding here, this is even slower than the magenta algorithm with 4 times more generations. \\
Note that \ref{plot:progress:memetics:10:linear:20} and \ref{plot:progress:memetics:50:tournament:5} spent most time on checking $t_\text{GA}$ generations with the same value.
}\label{fig:exp:gen:progress}
\end{figure}

    \subsubsection{Properties of the local search algorithm}

    Table~\ref{tab:exp:localsearch} shows the results over $10$ runs of the local search method with either the first-improvement policy (Alg.~\ref{alg:localSearch:first}) or the best-improvement policy (Alg.~\ref{alg:localSearch:best}). In each run, both of the local search methods start from the same randomly generated configuration. The highest optimal value found during the runs, the average number of linear programs solved and the average number of iterations are reported for each method.

    The first-improvement policy, which selects the first neighboring configuration with a higher objective value, was able to find a better solution for all instances. The numbers of iterations suggest that the best improvement policy quickly ends in a local maximum. On the other hand, the random nature of choosing the first improving neighbor seems to make the algorithm able to reach a solution of a decent quality. 

    The number of iterations of the first-improvement policy is (on average) greater than the total number of supplies and demands (40). This means that some of the supplies or demands are modified repeatedly. Also, the average numbers of LPs indicates that the average number of examined neighbors in an iteration of local search with the first-improvement policy is $\approx12$ (from the 39 neighbors available). For the best-improvement policy, the average number of LPs per iteration is $\approx39$ (as expected).

\newcolumntype{K}[0]{>{\centering\arraybackslash}p{40pt}}
    
\begin{table}[h!]
    \begin{tabular}{cKKK@{\hskip20pt}KKK}
    \hline
& \multicolumn{3}{c}{Local search with first-improvement} & \multicolumn{3}{c}{Local search with best-improvement} \\
& Obj. value & LPs & Iterations & Obj. value & LPs & Iterations \\ \hline
1 & \bf 9405	&	535.0	&	42.1	&	8175	&	90.7	&	2.3 \\
2 & \bf 9080	&	523.7	&	41.2	&	7845	&	94.6	&	2.4 \\
3 & \bf 9405	&	535.0	&	42.1	&	8175	&	90.7	&	2.3 \\
4 & \bf 9130	&	499.0	&	37.2	&	7810	&	110.2	&	2.8 \\
5 & \bf 9420	&	599.1	&	47.4	&	7340	&	110.2	&	2.8 \\
6 & \bf 10320	&	679.4	&	53.5	&	7825	&	98.5	&	2.5 \\
7 & \bf 8700	&	487.2	&	37.8	&	7780	&	102.4	&	2.6 \\
8 & \bf 9260	&	575.0	&	44.9	&	7370	&	110.2	&	2.8 \\
9 & \bf 9885	&	546.2	&	40.4	&	8285	&	110.2	&	2.8 \\
10 & \bf 9315	&	548.3	&	44.7	&	7710	&	102.4	&	2.6 \\
\hline
\end{tabular}
\caption{The highest worst optimal value found by the local search algorithm with the first-improvement and the best-improvement policy over $10$ runs. The average number of linear programs solved and the average number of iterations for the algorithms are reported in the respective columns.}\label{tab:exp:localsearch}
\end{table}
	
	\section{Conclusion}\label{sec:Conclusion}
	Interval transportation problems provide a mathematical model for handling and solving transportation problems affected by uncertainty, in which the values of supply, demand, and the transportation costs can vary within given interval ranges. In this paper, we addressed the problem of computing the worst finite optimal value of interval transportation problems, which was formerly proved to be NP-hard. Since an exact calculation of the value can be computationally demanding, we focused on heuristic approaches for approximating the value.

    Building on a recent result showing a decomposition of the problem into a set of quasi-extreme scenarios, we proposed a memetic algorithm for finding a lower bound on the worst finite optimal value of an interval transportation problem. We first derived a stronger result restricting the decomposition to the set of balanced quasi-extreme scenarios, in which the total supply and total demand are equal. Based on this theoretical result, we defined an encoding of the scenarios and proposed a local search algorithm on the set of encodings, as well as a genetic algorithm with suitable mutation and crossover operators. Combining the two methods, we formulated a memetic algorithm using both evolutionary improvement and learning through local search.

    The conducted computational experiments illustrate that the proposed memetic algorithm integrates the efficiency of the local search method and the quality of the solutions found by the population-based genetic algorithm. This leads to a method for finding better lower bounds, while also maintaining a fast computation time. Furthermore, the memetic algorithm has been able to find the new best results for 2 of instances (out of 10), while for the other instances, it finds the best known solutions faster than the competitors.

\backmatter

\section*{Declarations}

\bmhead{Funding}
E. Radov\'{a} Garajov\'{a} and M. Rada were supported by the Czech Science Foundation under grant 23-07270S.

\bmhead{Competing interests}
The authors have no competing interests to declare that are relevant to the content of this article.

\bmhead{Data availability} The data generated by the authors are available in the GitHub repository \citep{Rada:Github:2025}. Data from previously published papers of other authors are not disclosed here.

\bmhead{Code availability} The source code of the implementation is available in the GitHub repository \citep{Rada:Github:2025}.

\bmhead{Author contribution} Both authors contributed equally to this work.

\bibliography{memetic-alg}

@misc{Rada:Github:2025,
  author = {Rada, M. and Radová Garajová, E.},
  title = {Interval Transportation Problems},
  year = {2025},
  publisher = {GitHub},
  journal = {GitHub repository},
  url = {https://github.com/polyraroh/itp}
}

@book{Sivanandam:IntroductionGeneticAlgorithms:2008,
  title = {Introduction to {{Genetic Algorithms}}},
  author = {Sivanandam, S. N. and Deepa, S. N.},
  year = {2008},
  publisher = {Springer},
  address = {Berlin, Heidelberg},
  doi = {10.1007/978-3-540-73190-0},
  urldate = {2025-02-14},
  copyright = {http://www.springer.com/tdm},
  isbn = {978-3-540-73189-4},
  langid = {english}
}

@article{Carrabs:ImprovedHeuristicApproach:2021a,
  title = {An Improved Heuristic Approach for the Interval Immune Transportation Problem},
  author = {Carrabs, Francesco and Cerulli, Raffaele and D'Ambrosio, Ciriaco and Della Croce, Federico and Gentili, Monica},
  year = {2021},
  month = oct,
  journal = {Omega},
  volume = {104},
  pages = {102492},
  issn = {0305-0483},
  doi = {10.1016/j.omega.2021.102492},
  urldate = {2024-04-03}
}

@inproceedings{Cerulli:BestWorstValues:2017,
  title = {Best and {{Worst Values}} of the {{Optimal Cost}} of the {{Interval Transportation Problem}}},
  booktitle = {Optimization and {{Decision Science}}: {{Methodologies}} and {{Applications}}},
  author = {Cerulli, R. and D'Ambrosio, C. and Gentili, M.},
  editor = {Sforza, Antonio and Sterle, Claudio},
  year = {2017},
  series = {Springer {{Proceedings}} in {{Mathematics}} \& {{Statistics}}},
  pages = {367--374},
  publisher = {Springer},
  address = {Cham},
  doi = {10.1007/978-3-319-67308-0_37},
  isbn = {978-3-319-67308-0},
  langid = {english}
}

@article{DAmbrosio:OptimalValueRange:2020,
  title = {The Optimal Value Range Problem for the {{Interval}} (Immune) {{Transportation Problem}}},
  author = {D'Ambrosio, C. and Gentili, M. and Cerulli, R.},
  year = {2020},
  month = sep,
  journal = {Omega},
  volume = {95},
  pages = {102059},
  issn = {0305-0483},
  doi = {10.1016/j.omega.2019.04.002},
  urldate = {2023-05-23},
  langid = {english}
}

@article{Garajova:ComplexityComputingWorst:2024,
  title = {Complexity of Computing the Worst Optimal Value of Interval Transportation Problems},
  author = {Garajov{\'a}, Elif and Rada, Miroslav},
  year = {2024},
  month = nov,
  journal = {Central European Journal of Operations Research},
  issn = {1613-9178},
  doi = {10.1007/s10100-024-00947-8},
  urldate = {2025-01-22},
  langid = {english}
}

@article{Garajova:IntervalTransportationProblem:2023,
  title = {Interval Transportation Problem: Feasibility, Optimality and the Worst Optimal Value},
  shorttitle = {Interval Transportation Problem},
  author = {Garajov{\'a}, Elif and Rada, Miroslav},
  year = {2023},
  month = sep,
  journal = {Central European Journal of Operations Research},
  volume = {31},
  number = {3},
  pages = {769--790},
  issn = {1613-9178},
  doi = {10.1007/s10100-023-00841-9},
  urldate = {2023-06-30},
  langid = {english}
}

@inproceedings{Garajova:QuasiextremeReductionInterval:2024,
  title = {A {{Quasi-extreme Reduction}} for~{{Interval Transportation Problems}}},
  booktitle = {Dynamics of {{Information Systems}}},
  author = {Garajov{\'a}, Elif and Rada, Miroslav},
  editor = {Moosaei, Hossein and Hlad{\'i}k, Milan and Pardalos, Panos M.},
  year = {2024},
  pages = {83--92},
  publisher = {Springer Nature Switzerland},
  address = {Cham},
  doi = {10.1007/978-3-031-50320-7_6},
  isbn = {978-3-031-50320-7},
  langid = {english}
}

@article{Hladik:WorstCaseFinite:2018,
  title = {The Worst Case Finite Optimal Value in Interval Linear Programming},
  author = {Hlad{\'i}k, Milan},
  year = {2018},
  journal = {Croatian Operational Research Review},
  volume = {9},
  number = {2},
  pages = {245--254},
  issn = {1848-9931},
  doi = {10.17535/crorr.2018.0019}
}

@article{Hoppmann-Baum:ComplexityComputingMaximum:2022,
  title = {On the {{Complexity}} of {{Computing Maximum}} and {{Minimum Min-Cost-Flows}}},
  author = {{Hoppmann-Baum}, Kai},
  year = {2022},
  journal = {Networks},
  volume = {79},
  number = {2},
  pages = {236--248},
  issn = {1097-0037},
  doi = {10.1002/net.22060},
  urldate = {2023-06-30},
  copyright = {{\copyright} 2021 The Authors. Networks published by Wiley Periodicals LLC.},
  langid = {english}
}

@article{Juman:HeuristicSolutionTechnique:2014,
  title = {A Heuristic Solution Technique to Attain the Minimal Total Cost Bounds of Transporting a Homogeneous Product with Varying Demands and Supplies},
  author = {Juman, Z. A. M. S. and Hoque, M. A.},
  year = {2014},
  month = nov,
  journal = {European Journal of Operational Research},
  volume = {239},
  number = {1},
  pages = {146--156},
  issn = {0377-2217},
  doi = {10.1016/j.ejor.2014.05.004},
  urldate = {2023-06-28},
  langid = {english}
}

@article{Liu:TotalCostBounds:2003,
  title = {The Total Cost Bounds of the Transportation Problem with Varying Demand and Supply},
  author = {Liu, Shiang-Tai},
  year = {2003},
  month = aug,
  journal = {Omega},
  volume = {31},
  number = {4},
  pages = {247--251},
  issn = {0305-0483},
  doi = {10.1016/S0305-0483(03)00054-9},
  urldate = {2023-06-28},
  langid = {english}
}

@incollection{Mohammadi:IntervalLinearProgramming:2020,
  title = {Interval {{Linear Programming}}: {{Optimal Value Range}}},
  shorttitle = {Interval {{Linear Programming}}},
  booktitle = {Encyclopedia of {{Optimization}}},
  author = {Mohammadi, Mohsen and Hlad{\'i}k, Milan and Gentili, Monica},
  editor = {Pardalos, Panos M. and Prokopyev, Oleg A.},
  year = {2020},
  pages = {1--11},
  publisher = {Springer International Publishing},
  address = {Cham},
  doi = {10.1007/978-3-030-54621-2_718-1},
  urldate = {2025-01-22},
  isbn = {978-3-030-54621-2},
  langid = {english}
}

@incollection{Moscato:ModernIntroductionMemetic:2010,
  title = {A {{Modern Introduction}} to {{Memetic Algorithms}}},
  booktitle = {Handbook of {{Metaheuristics}}},
  author = {Moscato, Pablo and Cotta, Carlos},
  editor = {Gendreau, Michel and Potvin, Jean-Yves},
  year = {2010},
  pages = {141--183},
  publisher = {Springer US},
  address = {Boston, MA},
  doi = {10.1007/978-1-4419-1665-5_6},
  urldate = {2025-01-22},
  isbn = {978-1-4419-1665-5},
  langid = {english}
}

@incollection{Rohn:IntervalLinearProgramming:2006,
  title = {Interval Linear Programming},
  booktitle = {Linear {{Optimization Problems}} with {{Inexact Data}}},
  author = {Rohn, J.},
  editor = {Fiedler, M. and Nedoma, J. and Ram{\'i}k, J. and Rohn, J. and Zimmermann, K.},
  year = {2006},
  pages = {79--100},
  publisher = {Springer US},
  address = {Boston, MA},
  doi = {10.1007/0-387-32698-7_3},
  urldate = {2023-05-23},
  isbn = {978-0-387-32698-6},
  langid = {english}
}

@article{Xie:UpperBoundMinimal:2017,
  title = {An Upper Bound on the Minimal Total Cost of the Transportation Problem with Varying Demands and Supplies},
  author = {Xie, Fanrong and Butt, Muhammad Munir and Li, Zuoan and Zhu, Linzhi},
  year = {2017},
  month = apr,
  journal = {Omega},
  volume = {68},
  pages = {105--118},
  issn = {0305-0483},
  doi = {10.1016/j.omega.2016.06.007},
  urldate = {2023-06-28},
  langid = {english}
}

\end{document}